\documentclass[11pt]{article}
\usepackage[utf8]{inputenc}
\usepackage{amsthm, amsmath, amssymb, amsfonts, url, booktabs, tikz, setspace, fancyhdr, bm}
\usepackage{hyperref}
\usepackage{geometry}
\usepackage{hyperref, enumerate}
\usepackage[shortlabels]{enumitem}
\usepackage{enumitem}
\usepackage[english]{babel}
\usepackage[capitalise]{cleveref}
\usepackage{bbm,tkz-graph,subcaption}
\usepackage{csquotes}
\usepackage{mathrsfs}
\usepackage{mathabx}
\usetikzlibrary{patterns}
\usetikzlibrary{shapes}
\usepackage{bbm,wrapfig}
\usepackage{xspace}

\usepackage{algorithm}
\usepackage{algorithmic}
\usepackage[utf8]{inputenc} 
\usepackage[T1]{fontenc}
\usepackage{amsfonts}
\usepackage{amscd}
\usepackage{graphicx}
\usepackage{enumitem}
\usepackage{array, booktabs, makecell}
\usepackage{verbatim}
\usepackage{hyperref}
\usepackage{amsmath,caption}
\usepackage{url,pdfpages,xcolor,framed,color}
\usepackage{todonotes}
\usepackage{comment}

\newtheorem{thr}{Theorem}[section]

\newtheorem{lem}[thr]{Lemma}

\newtheorem{conj}[thr]{Conjecture}
\theoremstyle{definition}

\newtheorem*{defi*}{Definition}
\newtheorem{cor}[thr]{Corollary}

\newtheorem{obs}[thr]{Observation}
\newtheorem{prob}[thr]{Problem}

\newcommand*{\myproofname}{Proof}

\tikzstyle{P} = [draw, circle, black, fill, inner sep = 0pt, minimum width = 3pt]
\tikzstyle{every loop} = []

\usetikzlibrary{calc}

\title{Counterexamples to the Albertson-Berman conjecture: minimum order, connectivity and an improved ratio bound}

\date{}
\author{
{\sc Wouter CAMES VAN BATENBURG}$^{\ast}$,
{\sc Jan GOEDGEBEUR}$^{\dagger,\ddagger}$,
{\sc Jorik JOOKEN}$^{\diamond,\dagger}$\\[1mm]
{\small $^{\ast}$D\'epartement d'Informatique,
Universit\'e libre de Bruxelles, Belgium}\\
{\small $^{\dagger}$Department of Computer Science, KU Leuven Kulak,
8500 Kortrijk, Belgium.}\\
{\small $^{\ddagger}$Department of Mathematics, Computer Science and Statistics,
Ghent University, 9000 Ghent, Belgium.}\\
{\small $^{\diamond}$Mathematical Institute, Leiden University, Einsteinweg 55,
2333 CC Leiden, The Netherlands.}
}

\begin{document}
\maketitle

\begin{abstract}
\noindent In 1979, Albertson and Berman conjectured that every planar graph $G$ contains an induced forest of order at least $|V(G)|/2$. This long-standing conjecture was recently disproved by several explicit counterexamples, which naturally led to several extremal and structural questions that we answer. We combine mathematical arguments and exhaustive computations to show that the minimum order of a counterexample is $29$. We also construct infinitely many $4$-connected $5$-edge-connected counterexamples (and show that the unique such counterexample of minimum order has order $41$), whereas previously all known counterexamples had vertex-connectivity at most $3$. Furthermore, we construct an infinite family of planar graphs on $n$ vertices whose maximum induced forests have order at most $\frac{25}{52}n$, thereby improving the previous best upper bound. This family also yields infinitely many counterexamples (for every integer $d \geq 7$) to a conjecture of Chappell and Pelsmajer concerning induced forests of maximum degree at most $d$.


\end{abstract}

\section{Introduction}

All graphs discussed in this paper are simple non-empty graphs. For a graph $G$, let $a(G)$ be the maximum number of vertices among all induced forests of $G$ and let
$$
c_{\mathcal P}:=\inf\left\{\frac{a(G)}{|V(G)|}: G \text{ is a planar graph}\right\}.
$$

By the Four Colour Theorem, a planar graph on $n$ vertices contains an independent set of order at least $n/4$. Cranston and Rabern~\cite{CR16} note that Erd\H{o}s asked in 1968 already whether this consequence could be proved more easily (i.e., without relying on the Four Colour Theorem). Cranston and Rabern~\cite{CR16} improved the best bound that does not rely on the Four Colour Theorem to $3n/13$. Every forest is bipartite and therefore an induced forest containing at least $n/2$ vertices would also yield an independent set containing at least $n/4$ vertices. Therefore, finding large induced forests can be viewed as a strengthening of this classical consequence of the Four Colour Theorem.

A classical result of Borodin~\cite{B79} states that every planar graph admits a proper $5$-colouring such that the union of every two colour classes induces a forest. This immediately yields the bound $c_{\mathcal P}\geq \frac{2}{5}$. It is trivial that $c_{\mathcal P}\leq 1/2$, since $a(K_4)=2$. In 1979, Albertson and Berman~\cite{AB79} made the following well-known conjecture.

\begin{conj}[Albertson-Berman conjecture~\cite{AB79}]
\label{conj:AB}
    We have $c_{\mathcal P}\geq 1/2$.
\end{conj}

This conjecture remained open for almost five decades and led to a large stream of research on large induced forests and related vertex-decomposition problems. The prominence of the Albertson-Berman conjecture can also be seen by its inclusion in the collections of open problems by West~\cite{WestProblems} and Mohar~\cite{MoharProblems}. 

We now briefly discuss several related papers. Alon, Mubayi and Thomas~\cite{AMT01} proved several lower bounds on $a(G)$ that are especially useful for sparse graphs. Several restricted classes of planar graphs were also studied. Borodin and Glebov~\cite{BG01} showed that the vertex set of every planar graph of girth at least $5$ can be partitioned into an induced forest and an independent set. Subsequently, this result was strengthened by Kawarabayashi and Thomassen~\cite{KT09}, who showed a list-colouring version thereof. Le~\cite{Le18} investigated the class of triangle-free planar graphs and showed that every such graph satisfies $\frac{a(G)}{|V(G)|} \geq \frac{5}{9}$. Bonamy, Kardo\v{s}, Kelly and Postle~\cite{BKKP20} studied the fractional vertex-arboricity of planar graphs and proposed fractional strengthenings of several conjectures about large induced forests, including the Albertson-Berman conjecture. More recently, an equivalent formulation of the Albertson-Berman conjecture in terms of robust connectivity of planar graphs was obtained by Bradshaw, Masa\v{r}\'ik, Novotn\'a and Stacho~\cite{BMNS22}.

In August 2026, the Albertson-Berman conjecture was finally disproved by several explicit counterexamples that were independently presented by different authors. Jung~\cite{J26} presented a planar triangulation on $31$ vertices in which all maximum induced forests have order $15$, yielding $c_{\mathcal P}\leq \frac{15}{31} \approx 0.4839$. Independently, Makarov~\cite{M26} presented two counterexamples on $39$ and $76$ vertices, yielding $c_{\mathcal P}\leq \frac{37}{76} \approx 0.4868< \frac{19}{39} \approx 0.4872$. These results naturally lead to the following extremal and structural questions that we deal with in the current paper: what is the minimum order of a counterexample to the Albertson-Berman conjecture, can counterexamples have stronger connectivity properties than the known counterexamples, and how small can $c_{\mathcal{P}}$ be?

In Section~\ref{sec:minOrder}, we show that the minimum order of a counterexample to the Albertson-Berman conjecture is $29$ and present two counterexamples that yield $c_{\mathcal P}\leq \frac{14}{29} \approx 0.4828$, whereas in Section~\ref{sec:4Connected} we present infinitely many $4$-connected $5$-edge-connected planar counterexamples and show that the unique such counterexample of minimum order has order $41$. In contrast, all previously known counterexamples to the Albertson-Berman conjecture have vertex-connectivity at most $3$. This positively answers Question~2 of Jung~\cite{J26}. In Section~\ref{sec:ratio2552}, we present an infinite family of counterexamples to the Albertson-Berman conjecture that yield $c_{\mathcal{P}} \leq \frac{25}{52} \approx 0.4808$ (this is the best bound so far) and our exhaustive computations show that no $4$-connected $5$-edge-connected counterexample on at most $41$ vertices yields a better bound. This family also yields infinitely many counterexamples (for every integer $d \geq 7$) to a conjecture of Chappell and Pelsmajer~\cite{CP13} concerning induced forests of maximum degree at most $d$. Finally, in Section~\ref{sec:openProblems} we close this paper by mentioning several natural open problems that spark our interest.

\section{The minimum order of a counterexample to the Albertson-Berman conjecture is $29$}
\label{sec:minOrder}
The $29$-vertex counterexamples presented at the end of this section were the starting point of this project and were found before the exhaustive computations described below. These computations were subsequently carried out to establish that no smaller counterexample exists.

We first make the following elementary observation, which allows us to restrict our attention to planar triangulations.

\begin{obs}
\label{obs:ToPlanarTriangulation}
    Let $G$ be a planar graph on $n$ vertices that is a counterexample to Conjecture~\ref{conj:AB}. Then there exists a planar triangulation $G'$ on $n$ vertices that is also a counterexample to Conjecture~\ref{conj:AB}.
\end{obs}

Indeed, one can obtain a planar triangulation $G'$ by adding edges to $G$ until every face is a triangle. Every induced forest of $G'$ is also an induced forest of $G$ and therefore $a(G') \leq a(G)$.

We now establish the following lemma.
\begin{lem}
\label{lem:order25And27}
    If there exists a counterexample to Conjecture~\ref{conj:AB} on at most $28$ vertices, then there exists a planar triangulation on $25$ or $27$ vertices that is a counterexample to Conjecture~\ref{conj:AB}.
\end{lem}
\begin{proof}
    Suppose $G_1$ is a counterexample on  $n \leq 28$ vertices. We have $a(G_1) < \frac{n}{2}$. Let $v \in V(G_1)$ be any vertex in $G_1$ and let $G_2$ be obtained by deleting $v$ from $G_1$. If $n$ is even, we have $a(G_2)\leq a(G_1)\leq \frac{n}{2}-1<\frac{n-1}{2}$ and therefore $G_2$ is again a counterexample. Hence, we may assume there exists a counterexample $H$ on $m \leq 27$ vertices, where $m$ is odd.
    
    If $m \in \{25,27\}$, the lemma follows from applying Observation~\ref{obs:ToPlanarTriangulation} to $H$. If $m \notin \{25,27\}$, we can obtain a counterexample $H'$ on $m+4$ vertices that is a planar triangulation by taking the disjoint union of $K_4$ and $H$ and applying Observation~\ref{obs:ToPlanarTriangulation}, since $a(H+K_4)=a(H)+a(K_4)=a(H)+2<\frac{m+4}{2}$. By repeatedly applying this argument, the lemma follows.
\end{proof}

We now recall a classical theorem due to Stein~\cite{S71}, relating induced forests in planar triangulations to Hamiltonian cycles in their dual.

\begin{thr}[Stein~\cite{S71}]
\label{th:nonHamiltonianDual}
    Let $G$ be a planar triangulation and let $G^*$ be its dual. Then there exists a partition of $V(G)$ into two sets $V_1$ and $V_2$ such that both $G[V_1]$ and $G[V_2]$ are forests if and only if $G^*$ is Hamiltonian.
\end{thr}

In particular, if $G$ is a planar triangulation whose dual $G^*$ is Hamiltonian, Theorem~\ref{th:nonHamiltonianDual} implies that $G$ is not a counterexample to Conjecture~\ref{conj:AB}. Hence, together with Lemma~\ref{lem:order25And27}, this motivates the study of planar triangulations on $25$ or $27$ vertices whose dual is nonhamiltonian (i.e., the duals are precisely the nonhamiltonian cubic planar $3$-connected graphs on $46$ and $50$ vertices, respectively).

Holton and McKay~\cite{HM88} showed that the smallest nonhamiltonian cubic planar $3$-connected graphs have $38$ vertices and McKay~\cite{McKayCensus} generated an exhaustive list of all nonhamiltonian cubic planar $3$-connected graphs without faces of size $3$ on at most $48$ vertices. Using the generator \texttt{plantri}~\cite{BM07} and a filter for testing Hamiltonicity, we extended the latter list to $50$ vertices. More specifically, we generated all $16\,747\,182\,732\,792$ cubic planar $3$-connected graphs without faces of size $3$ on $50$ vertices and determined that exactly $441\,578$ of them are nonhamiltonian. This computation required approximately 10 CPU years on a cluster.

To also obtain the nonhamiltonian cubic planar $3$-connected graphs that have a face of size $3$, we use a standard operation. In particular, if $G$ is a nonhamiltonian cubic planar $3$-connected graph on $n$ vertices having a face of size $3$, then contracting that face yields a nonhamiltonian cubic planar $3$-connected graph $G'$ on $n-2$ vertices. Hence, if we repeatedly apply this operation, we obtain a nonhamiltonian cubic planar $3$-connected graph without faces of size $3$. This means that the original graph (with faces of size $3$) can be recovered by repeatedly expanding a vertex into a triangle. This allowed us to generate all pairwise non-isomorphic nonhamiltonian cubic planar $3$-connected graphs on at most $50$ vertices. In Table~\ref{tab:counts}, we summarize the total number of such graphs as well as the total number of such graphs without faces of size $3$.

\begin{table}[ht]
    \centering
    \begin{tabular}{c|r|r}
        $n$
        & \multicolumn{1}{c|}{Without faces of size $3$}
        & \multicolumn{1}{c}{Total} \\
        \hline
        $38$ & $6$ & $6$ \\
        $40$ & $37$ & $191$ \\
        $42$ & $277$ & $5\,065$ \\
        $44$ & $1\,732$ & $98\,415$ \\
        $46$ & $11\,204$ & $1\,600\,259$ \\
        $48$ & $70\,614$ & $22\,818\,523$ \\
        $50$ & $441\,578$ & $295\,835\,441$
    \end{tabular}
    \caption{The number of pairwise non-isomorphic nonhamiltonian cubic planar $3$-connected graphs on $n$ vertices (third column) and the number of such graphs without faces of size $3$ (second column).}
    \label{tab:counts}
\end{table}

By taking the duals of all nonhamiltonian cubic planar $3$-connected graphs on $46$ and $50$ vertices and calculating $a(G)$\footnote{In particular, we used the algorithm of Iwata and Imanishi~\cite{I17,II16} to compute the \textit{feedback vertex set number} $\mathrm{fvs}(G)$ (note that $a(G) = |V(G)| - \mathrm{fvs}(G)$).}, we obtain the following.

\begin{obs}
    There exists no counterexample on at most $28$ vertices to Conjecture~\ref{conj:AB}.
\end{obs}

Finally, we present two counterexamples on $29$ vertices to Conjecture~\ref{conj:AB} that can be obtained by appropriately combining two gadgets. We note that all small graphs from this paper can be downloaded from the \textit{House of Graphs}~\cite{CDG} by searching for the keyword ``Albertson-Berman''. Let $Q$ be the $14$-vertex graph shown in Fig.~\ref{fig:gadget14} and let $R$ be the $17$-vertex graph shown in Fig.~\ref{fig:gadget17}. Call the edge $xy$ in $Q$ and the edge $bc$ in $R$ \textit{distinguished edges}. The gadget $Q$ is, up to an isomorphism respecting the distinguished edge, the same $14$-vertex gadget that plays a central role in the counterexamples of Jung~\cite{J26} and Makarov~\cite{M26}. Jung's $31$-vertex counterexample uses two copies of $Q$, while Makarov's $39$-vertex counterexample is obtained by attaching three copies of $Q$ to the edges of a triangle. Both gadgets $Q$ and $R$ are closely related to the icosahedron graph. In $Q$, let $p$ and $q$ be the two common neighbours of $x$ and $y$. Deleting $x$ and $y$ and adding the edge $pq$ yields the icosahedron graph. In $R$, let $N$ be the subgraph induced by the closed neighbourhood of $\{b,c\}$. Besides $b$ and $c$, the graph $N$ contains four vertices $t,w,p,s$ such that $tbcw$ is a path, $p$ and $s$ are nonadjacent, and both $p$ and $s$ are adjacent to all four vertices of this path. Contracting $N$ to a single vertex yields the icosahedron graph. This also gives some intuition for why modifications of the icosahedron are useful here: if $I$ denotes the icosahedron graph, then $\alpha(I)=3=|V(I)|/4$ and $a(I)=6=|V(I)|/2$.

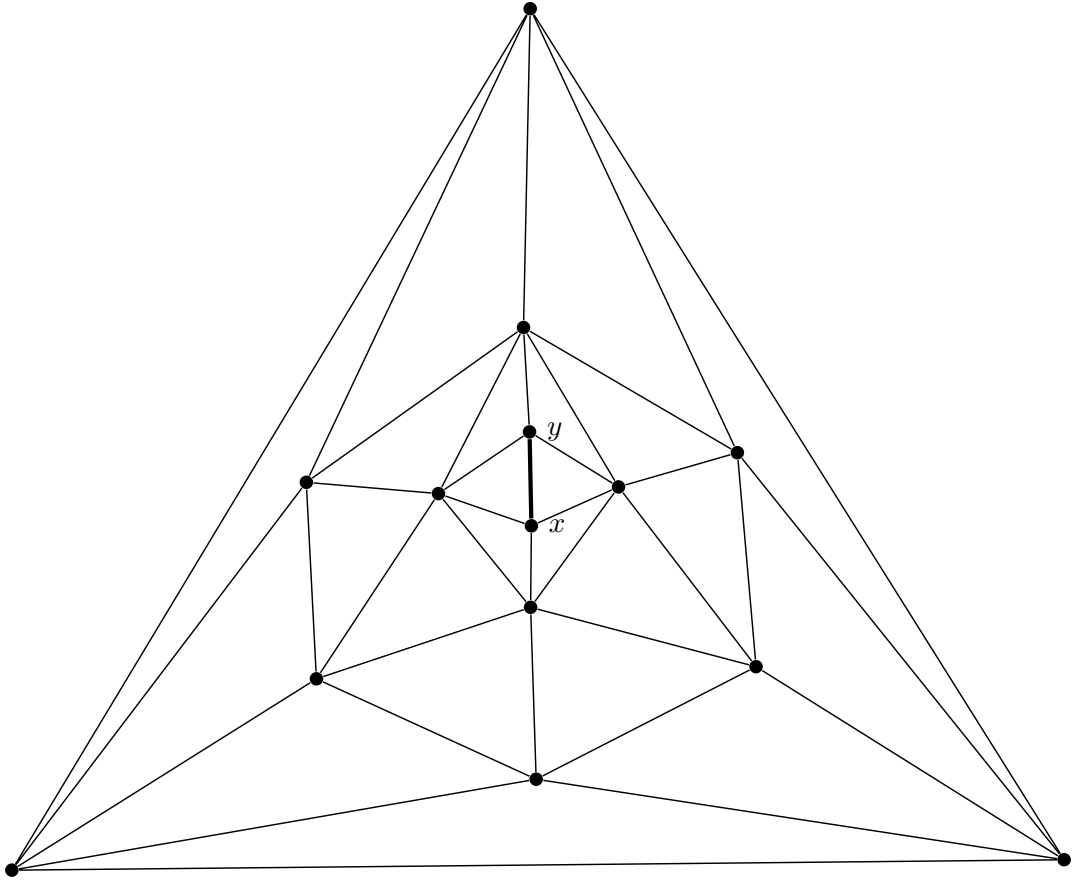
\begin{figure}[ht]
    \centering

    \begin{tikzpicture}[xscale=5.8,yscale=5.8]

    \tikzset{
        graphvertex/.style={
            circle,
            fill=black,
            inner sep=1.8pt
        },
        graphedge/.style={
            line width=0.55pt
        },
        distinguishededge/.style={
            line width=1.6pt
        }
    }

    \node[graphvertex,label=right:$x$]        (v0)  at (-0.0150,-0.19714) {};
    \node[graphvertex,label=right:$y$]        (v1)  at (-0.01929,0.01714) {};

    \node[graphvertex]                        (v2)  at ( 1.2000,-0.95864) {};
    \node[graphvertex]                        (v3)  at (-1.2000,-0.98239) {};
    \node[graphvertex]                        (v4)  at ( 0.45473,-0.03029) {};
    \node[graphvertex]                        (v5)  at ( 0.49731,-0.51834) {};
    \node[graphvertex]                        (v6)  at (-0.50531,-0.54607) {};
    \node[graphvertex]                        (v7)  at (-0.52835,-0.09787) {};
    \node[graphvertex]                        (v8)  at (-0.01777, 0.98239) {};
    \node[graphvertex]                        (v9)  at (-0.00418,-0.77504) {};
    \node[graphvertex,label=above right:]        (v10) at (-0.03301,0.25523) {};
    \node[graphvertex,label=below right:]        (v11) at (-0.01685,-0.38297) {};
    \node[graphvertex,label=above:]        (v12) at (0.18369,-0.10837) {};
    \node[graphvertex,label=above left:]         (v13) at (-0.22703,-0.12360) {};

    \draw[graphedge]
    (v0)--(v11)
    (v0)--(v12)
    (v0)--(v13)

    (v1)--(v10)
    (v1)--(v12)
    (v1)--(v13)

    (v2)--(v3)
    (v2)--(v4)
    (v2)--(v5)
    (v2)--(v8)
    (v2)--(v9)

    (v3)--(v6)
    (v3)--(v7)
    (v3)--(v8)
    (v3)--(v9)

    (v4)--(v5)
    (v4)--(v8)
    (v4)--(v10)
    (v4)--(v12)

    (v5)--(v9)
    (v5)--(v11)
    (v5)--(v12)

    (v6)--(v7)
    (v6)--(v9)
    (v6)--(v11)
    (v6)--(v13)

    (v7)--(v8)
    (v7)--(v10)
    (v7)--(v13)

    (v8)--(v10)

    (v9)--(v11)

    (v10)--(v12)
    (v10)--(v13)

    (v11)--(v12)
    (v11)--(v13);


    \draw[distinguishededge] (v0)--(v1);

    \end{tikzpicture}

    \caption{The $14$-vertex gadget $Q$. The distinguished edge is $xy$.}
    \label{fig:gadget14}
\end{figure}

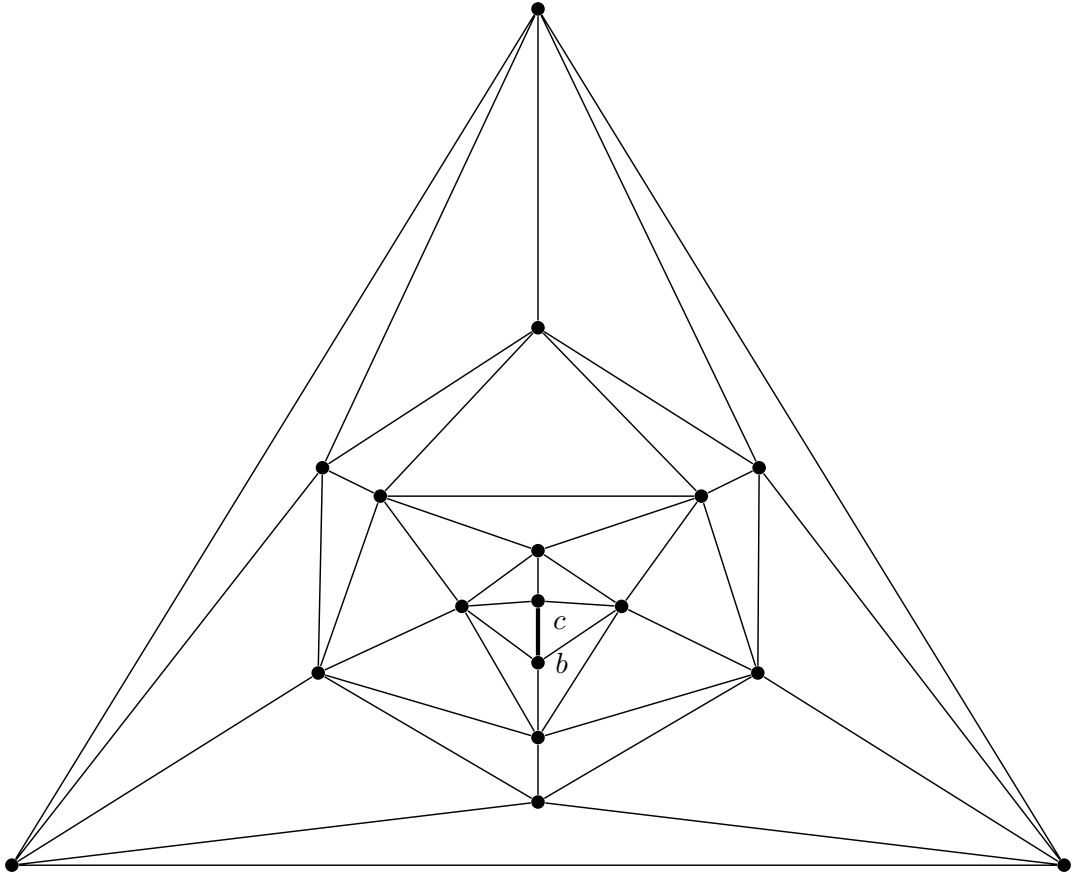
\begin{figure}[ht]
    \centering

    \begin{tikzpicture}[xscale=5.8,yscale=5.8]

    \tikzset{
        graphvertex/.style={
            circle,
            fill=black,
            inner sep=1.8pt
        },
        graphedge/.style={
            line width=0.55pt
        },
        distinguishededge/.style={
            line width=1.6pt
        }
    }

    \node[graphvertex,label=right:$b$]         (v0)  at (0,-0.51472) {};
    \node[graphvertex,label=below right:$c$]         (v1)  at (0,-0.37373) {};
    \node[graphvertex]                        (v2)  at (0,-0.83241) {};
    \node[graphvertex]                        (v3)  at ( 1.2000,-0.97645) {};
    \node[graphvertex]                        (v4)  at (-1.2000,-0.97645) {};
    \node[graphvertex]                        (v5)  at (0,0.24929) {};
    \node[graphvertex]                        (v6)  at ( 0.50440,-0.07002) {};
    \node[graphvertex]                        (v7)  at (-0.49154,-0.07002) {};
    \node[graphvertex]                        (v8)  at (0,0.97645) {};
    \node[graphvertex,label=above:]        (v9)  at (0,-0.25916) {};
    \node[graphvertex]                        (v10) at ( 0.50131,-0.53814) {};
    \node[graphvertex]                        (v11) at (-0.50131,-0.53814) {};
    \node[graphvertex]                        (v12) at ( 0.37250,-0.13478) {};
    \node[graphvertex]                        (v13) at (-0.35965,-0.13478) {};
    \node[graphvertex,label=below right:]        (v14) at (0,-0.68551) {};
    \node[graphvertex,label=above:]        (v15) at (0.19064,-0.38624) {};
    \node[graphvertex,label=above:]         (v16) at (-0.17349,-0.38624) {};

    \draw[graphedge]
    (v0)--(v14)
    (v0)--(v15)
    (v0)--(v16)

    (v1)--(v9)
    (v1)--(v15)
    (v1)--(v16)

    (v2)--(v3)
    (v2)--(v4)
    (v2)--(v10)
    (v2)--(v11)
    (v2)--(v14)

    (v3)--(v4)
    (v3)--(v6)
    (v3)--(v8)
    (v3)--(v10)

    (v4)--(v7)
    (v4)--(v8)
    (v4)--(v11)

    (v5)--(v6)
    (v5)--(v7)
    (v5)--(v8)
    (v5)--(v12)
    (v5)--(v13)

    (v6)--(v8)
    (v6)--(v10)
    (v6)--(v12)

    (v7)--(v8)
    (v7)--(v11)
    (v7)--(v13)

    (v9)--(v12)
    (v9)--(v13)
    (v9)--(v15)
    (v9)--(v16)

    (v10)--(v12)
    (v10)--(v14)
    (v10)--(v15)

    (v11)--(v13)
    (v11)--(v14)
    (v11)--(v16)

    (v12)--(v13)
    (v12)--(v15)

    (v13)--(v16)

    (v14)--(v15)
    (v14)--(v16);

    \draw[distinguishededge] (v0)--(v1);

    \end{tikzpicture}

    \caption{The $17$-vertex gadget $R$. The distinguished edge is $bc$.}
    \label{fig:gadget17}
\end{figure}

For a graph $J\in\{Q,R\}$ with distinguished edge $uv$ and $X \subseteq\{u,v\}$, we define
$$
p_J(X):=\max\bigl\{|F|:F\subseteq V(J),\ J[F]\text{ is a forest and }
F\cap\{u,v\}=X\bigr\}.
$$
We computationally verified that
$$\bigl(p_Q(\emptyset),p_Q(\{x\}),p_Q(\{y\}),p_Q(\{x,y\})\bigr)
=(6,7,7,7)$$ 
and 
$$\bigl(p_R(\emptyset),p_R(\{b\}),p_R(\{c\}),p_R(\{b,c\})\bigr)
=(8,8,8,9).$$
These values are not essentially computational: using the descriptions of $Q$ and $R$ as modifications of the icosahedron graph given above, they can also be verified by a short case analysis. To the best of our knowledge, the gadget $R$ is new. The extra ninth vertex in an induced forest of $R$ is possible only when both endpoints of the distinguished edge are present. This property is a key ingredient in our counterexamples on $29$ and $52$ vertices.

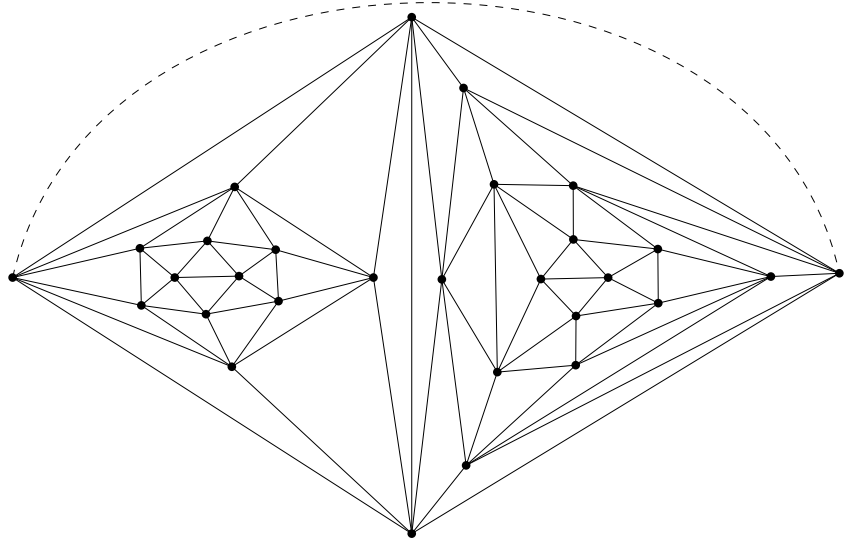
\begin{figure}
\centering
\begin{tikzpicture}[
    rotate=90,
    x=4.4cm,
    y=4.4cm,
    gedge/.style={draw=black,line width=.38pt},
    outeredge/.style={draw=black,line width=1.05pt},
    vertex/.style={circle,fill=black,inner sep=1.15pt}
]

\coordinate (v0) at (-0.000228607864,-0.590457640896);
\coordinate (v1) at (0.000000000000,0.712468478565);
\coordinate (v2) at (0.004279697963,0.519336869640);
\coordinate (v3) at (0.110205920461,0.614375341754);
\coordinate (v4) at (-0.110206051250,0.618655105112);
\coordinate (v5) at (-0.077033708520,-0.741206758428);
\coordinate (v6) at (-0.115670320362,-0.493950261279);
\coordinate (v7) at (0.088004128195,0.817252062347);
\coordinate (v8) at (0.083724430231,0.409048293049);
\coordinate (v9) at (-0.083724430231,0.812972298989);
\coordinate (v10) at (-0.070885205552,0.400488831728);
\coordinate (v11) at (0.085606927182,-0.739910021218);
\coordinate (v12) at (-0.004543737972,-0.388381901438);
\coordinate (v13) at (0.114593477320,-0.485928887884);
\coordinate (v14) at (-0.263451289431,-0.492997952436);
\coordinate (v15) at (0.569931939951,-0.155708165538);
\coordinate (v16) at (0.000000000000,1.200000000000);
\coordinate (v17) at (0.000000000000,0.115152303566);
\coordinate (v18) at (-0.268216972310,0.540925829564);
\coordinate (v19) at (0.272502686596,0.532354400993);
\coordinate (v20) at (0.012857142857,-1.285714285714);
\coordinate (v21) at (-0.005528937039,-0.090533873839);
\coordinate (v22) at (0.002991120992,-1.079790347605);
\coordinate (v23) at (-0.284344824781,-0.257354200237);
\coordinate (v24) at (0.276289798875,-0.485409051147);
\coordinate (v25) at (0.280276512935,-0.247348356109);
\coordinate (v26) at (-0.564781356086,-0.163507199234);
\coordinate (v27) at (0.782820323028,0.000000000000);
\coordinate (v28) at (-0.769963180170,0.000000000000);

\draw[gedge] (v0)--(v5);
\draw[gedge] (v0)--(v6);
\draw[gedge] (v0)--(v11);
\draw[gedge] (v0)--(v12);
\draw[gedge] (v0)--(v13);

\draw[gedge] (v1)--(v2);
\draw[gedge] (v1)--(v3);
\draw[gedge] (v1)--(v4);
\draw[gedge] (v1)--(v7);
\draw[gedge] (v1)--(v9);

\draw[gedge] (v2)--(v3);
\draw[gedge] (v2)--(v4);
\draw[gedge] (v2)--(v8);
\draw[gedge] (v2)--(v10);

\draw[gedge] (v3)--(v7);
\draw[gedge] (v3)--(v8);
\draw[gedge] (v3)--(v19);

\draw[gedge] (v4)--(v9);
\draw[gedge] (v4)--(v10);
\draw[gedge] (v4)--(v18);

\draw[gedge] (v5)--(v6);
\draw[gedge] (v5)--(v11);
\draw[gedge] (v5)--(v14);
\draw[gedge] (v5)--(v22);

\draw[gedge] (v6)--(v12);
\draw[gedge] (v6)--(v14);
\draw[gedge] (v6)--(v23);

\draw[gedge] (v7)--(v9);
\draw[gedge] (v7)--(v16);
\draw[gedge] (v7)--(v19);

\draw[gedge] (v8)--(v10);
\draw[gedge] (v8)--(v17);
\draw[gedge] (v8)--(v19);

\draw[gedge] (v9)--(v16);
\draw[gedge] (v9)--(v18);

\draw[gedge] (v10)--(v17);
\draw[gedge] (v10)--(v18);

\draw[gedge] (v11)--(v13);
\draw[gedge] (v11)--(v22);
\draw[gedge] (v11)--(v24);

\draw[gedge] (v12)--(v13);
\draw[gedge] (v12)--(v23);
\draw[gedge] (v12)--(v25);

\draw[gedge] (v13)--(v24);
\draw[gedge] (v13)--(v25);

\draw[gedge] (v14)--(v22);
\draw[gedge] (v14)--(v23);
\draw[gedge] (v14)--(v26);

\draw[gedge] (v15)--(v20);
\draw[gedge] (v15)--(v21);
\draw[gedge] (v15)--(v24);
\draw[gedge] (v15)--(v25);
\draw[gedge] (v15)--(v27);

\draw[gedge] (v16)--(v18);
\draw[gedge] (v16)--(v19);

\draw[gedge] (v17)--(v18);
\draw[gedge] (v17)--(v19);
\draw[gedge] (v17)--(v27);
\draw[gedge] (v17)--(v28);

\draw[gedge] (v18)--(v28);
\draw[gedge] (v19)--(v27);

\draw[gedge] (v20)--(v22);
\draw[gedge] (v20)--(v24);
\draw[gedge] (v20)--(v26);

\draw[gedge] (v21)--(v23);
\draw[gedge] (v21)--(v25);
\draw[gedge] (v21)--(v26);
\draw[gedge] (v21)--(v27);
\draw[gedge] (v21)--(v28);

\draw[gedge] (v22)--(v24);
\draw[gedge] (v22)--(v26);

\draw[gedge] (v23)--(v25);
\draw[gedge] (v23)--(v26);

\draw[gedge] (v24)--(v25);

\draw[gedge] (v26)--(v28);

\draw[gedge] (v16)--(v27);
\draw[gedge] (v16)--(v28);
\draw[gedge] (v20)--(v27);
\draw[gedge] (v20)--(v28);
\draw[gedge] (v27)--(v28);

\draw[gedge,dashed] (v16) .. controls (1.10,0.95) and (1.10,-1.05) .. (v20);

\foreach \i in {0,...,28}{
    \node[vertex] at (v\i) {};
}

\end{tikzpicture}

\caption{The $29$-vertex graph $G_{29}$ obtained by identifying $xy$ in $Q$ and $bc$ in $R$. Adding the dashed edge to $G_{29}$ results in $G'_{29}$.}
\label{fig:graph29}
\end{figure}

Let $G_{29}$ be the graph obtained by identifying $xy$ in $Q$ and $bc$ in $R$. Note that $G_{29}$ is planar (Fig.~\ref{fig:graph29} shows a planar embedding). For every induced forest of $G_{29}$, its restrictions to $Q$ and $R$ have the same intersection with the identified distinguished edge. Therefore, we obtain $a(G_{29}) \leq \max(6+8,7+8-1,7+8-1,7+9-2)=14$. By combining the induced forests in $Q$ and $R$ that both avoid the endpoints of the distinguished edge, we obtain $a(G_{29})=14$ and therefore $G_{29}$ is a counterexample to Conjecture~\ref{conj:AB}. The graph $G_{29}$ is a planar graph having $29$ vertices and $3\cdot 29-7=80$ edges. Up to isomorphism, there is a unique way to add one more edge to $G_{29}$ that results in a planar triangulation $G_{29}'$ and $G_{29}'$ is again a counterexample to Conjecture~\ref{conj:AB} for which $a(G_{29}')=14$ (see Fig.~\ref{fig:graph29}). We obtain:

\begin{thr}
    The minimum order of a counterexample to the Albertson-Berman conjecture is $29$.
\end{thr}

We now recall standard terminology that we will use in the next section. A graph with at least $k+1$ vertices is \textit{$k$-connected} if deleting any set of fewer than $k$ vertices leaves a connected graph. The \textit{vertex connectivity} of a graph is the largest integer $k$ for which it is $k$-connected. Similarly, a graph is \textit{$k$-edge-connected} if deleting any set of fewer than $k$ edges leaves a connected graph, and its \textit{edge connectivity} is the largest integer $k$ for which it is $k$-edge-connected. An edge cut is called \textit{cyclic} if its removal leaves at least two components containing a cycle. A graph is \textit{cyclically $k$-edge-connected} if every cyclic edge cut has size at least $k$.

We note that $G_{29}$ has vertex-connectivity $2$, whereas $G_{29}'$ has vertex-connectivity $3$.

\section{A minimum order counterexample with vertex-connectivity $4$ and edge-connectivity $5$ and an infinite family}
\label{sec:4Connected}

A natural question is to investigate whether the Albertson-Berman conjecture might be true under a stronger connectivity assumption. Indeed, Jung~\cite{J26} asked whether the Albertson-Berman conjecture fails for 4-connected planar graphs.

In the context of the second author's paper~\cite{GZ17}, we previously already generated all nonhamiltonian cubic planar $3$-connected graphs of girth $5$ on at most $70$ vertices and all such cyclically $4$-edge-connected graphs on at most $78$ vertices. For all of these graphs, we now calculated the dual and determined the maximum order of an induced forest.

By Observation~\ref{obs:ToPlanarTriangulation}, any $4$-connected $5$-edge-connected planar counterexample can be modified to a planar triangulation that is still $4$-connected and $5$-edge-connected and remains a counterexample by repeatedly adding edges. If $T$ is a planar triangulation on $n$ vertices, then its dual $T^*$ has $2n-4$ vertices. Moreover, $T$ is $4$-connected and $5$-edge-connected if and only if $T^*$ is cyclically $4$-edge-connected and has girth at least $5$. By performing the computations mentioned in the previous paragraph, we found exactly one counterexample $G_{41}$ on $\frac{78+4}{2}=41$ vertices that has vertex-connectivity $4$ and edge-connectivity $5$ and every maximum induced forest has $20$ vertices. The planar triangulation $G_{41}$ can again be obtained by appropriately combining the graphs $Q$ and $R$ (see Fig.~\ref{fig:graph41}).

Together with Theorem~\ref{th:nonHamiltonianDual} and our exhaustive computations, this yields the following.

\begin{thr}
\label{th:4Connected}
The minimum order of a $4$-connected $5$-edge-connected planar graph that is a counterexample to Conjecture~\ref{conj:AB} is $41$.
\end{thr}

By deleting an edge from this counterexample in all possible ways and verifying that this does not yield a new $4$-connected $5$-edge-connected counterexample, we obtain in fact the following slightly stronger conclusion.

\begin{thr}
\label{th:4ConnectedUniqueness}
There is a unique $4$-connected $5$-edge-connected planar graph on at most $41$ vertices that is a counterexample to Conjecture~\ref{conj:AB}.
\end{thr}

We note that this positively answers Question~2 from Jung~\cite{J26}, where it was asked whether the Albertson-Berman conjecture fails for $4$-connected planar graphs. 

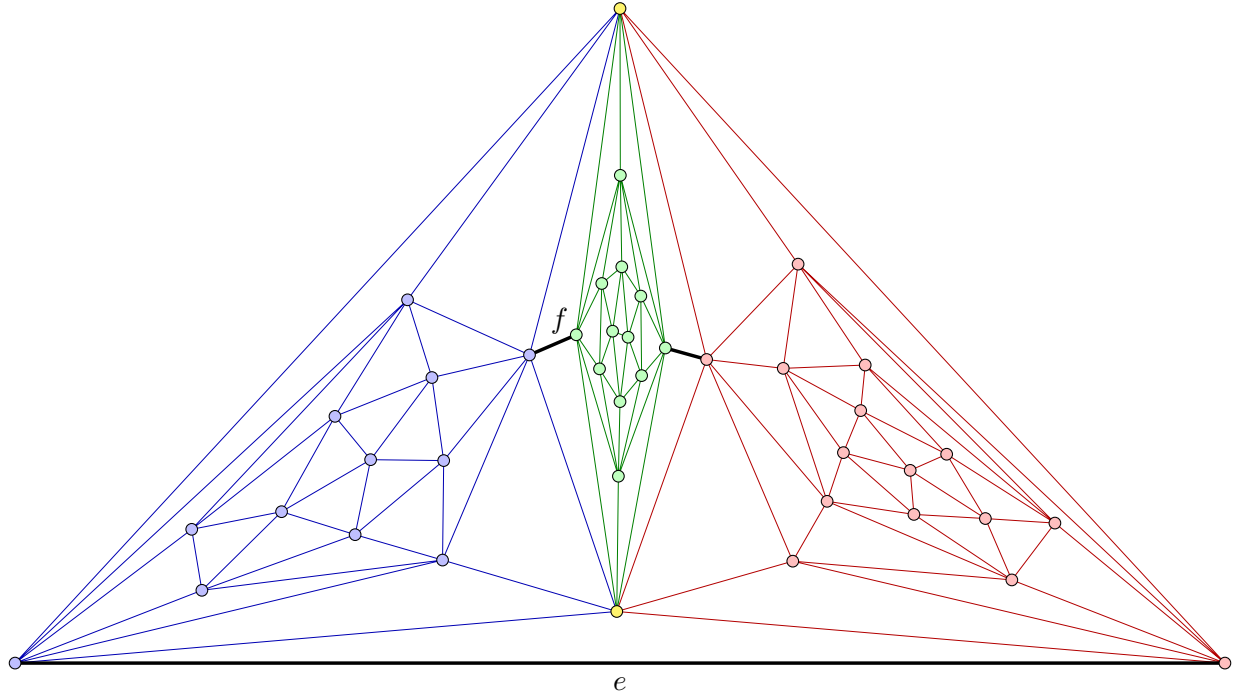
\begin{figure}
\centering
\begin{tikzpicture}[
xscale=1.6,
x=.1cm,
y=.1cm,
leftedge/.style={draw=blue!70!black,line width=.38pt},
middleedge/.style={draw=green!50!black,line width=.38pt},
rightedge/.style={draw=red!70!black,line width=.38pt},
crossedge/.style={draw=black,line width=1.3pt},
vertex/.style={circle,draw=black,fill=black,inner sep=1.55pt},
leftvertex/.style={vertex,fill=blue!25},
middlevertex/.style={vertex,fill=green!25},
rightvertex/.style={vertex,fill=red!25},
sharedvertex/.style={vertex,fill=yellow!70}
]

\coordinate (v41) at (73.986665,32.200000);
\coordinate (v40) at (80.199998,25.826667);
\coordinate (v39) at (76.999997,34.333334);
\coordinate (v38) at (69.893331,40.106667);
\coordinate (v37) at (68.466664,34.546667);
\coordinate (v36) at (74.293331,26.373334);
\coordinate (v35) at (50.666665,49.813334);
\coordinate (v34) at (49.386665,50.613334);
\coordinate (v33) at (49.999998,41.293334);
\coordinate (v32) at (82.386665,17.706667);
\coordinate (v31) at (85.946664,25.226667);
\coordinate (v30) at (70.266664,46.133333);
\coordinate (v29) at (63.493332,45.693334);
\coordinate (v28) at (67.119998,28.106667);
\coordinate (v27) at (51.786665,44.733334);
\coordinate (v26) at (51.719998,55.253334);
\coordinate (v25) at (50.146665,59.120000);
\coordinate (v24) at (48.493331,56.920000);
\coordinate (v23) at (48.306665,45.626667);
\coordinate (v22) at (49.866665,31.426667);
\coordinate (v21) at (64.279998,20.186667);
\coordinate (v20) at (28.106665,23.693334);
\coordinate (v19) at (15.440000,16.333334);
\coordinate (v18) at (99.999999,6.706668);
\coordinate (v17) at (64.719998,59.480001);
\coordinate (v16) at (57.159998,46.866667);
\coordinate (v15) at (53.746665,48.386667);
\coordinate (v14) at (50.026665,71.226668);
\coordinate (v13) at (46.386665,50.133334);
\coordinate (v12) at (49.733331,13.533335);
\coordinate (v11) at (35.333331,20.333334);
\coordinate (v10) at (35.426664,33.493334);
\coordinate (v9) at (29.386665,33.613334);
\coordinate (v8) at (22.026666,26.733334);
\coordinate (v7) at (14.599997,24.400000);
\coordinate (v6) at (0.000000,6.706668);
\coordinate (v5) at (49.999998,93.293331);
\coordinate (v4) at (42.519998,47.466667);
\coordinate (v3) at (34.453331,44.480000);
\coordinate (v2) at (26.439999,39.346667);
\coordinate (v1) at (32.439998,54.759999);

\draw[leftedge]
(v20)--(v8)
(v20)--(v9)
(v20)--(v10)
(v20)--(v11)
(v20)--(v19)
(v19)--(v6)
(v19)--(v7)
(v19)--(v8)
(v19)--(v11)
(v12)--(v4)
(v12)--(v6)
(v12)--(v11)
(v11)--(v4)
(v11)--(v6)
(v11)--(v10)
(v10)--(v3)
(v10)--(v4)
(v10)--(v9)
(v9)--(v2)
(v9)--(v3)
(v9)--(v8)
(v8)--(v2)
(v8)--(v7)
(v7)--(v1)
(v7)--(v2)
(v7)--(v6)
(v6)--(v1)
(v6)--(v5)
(v5)--(v1)
(v5)--(v4)
(v4)--(v1)
(v4)--(v3)
(v3)--(v1)
(v3)--(v2)
(v2)--(v1);

\draw[middleedge]
(v35)--(v25)
(v35)--(v26)
(v35)--(v27)
(v35)--(v33)
(v35)--(v34)
(v34)--(v23)
(v34)--(v24)
(v34)--(v25)
(v34)--(v33)
(v33)--(v22)
(v33)--(v23)
(v33)--(v27)
(v27)--(v15)
(v27)--(v22)
(v27)--(v26)
(v26)--(v14)
(v26)--(v15)
(v26)--(v25)
(v25)--(v14)
(v25)--(v24)
(v24)--(v13)
(v24)--(v14)
(v24)--(v23)
(v23)--(v13)
(v23)--(v22)
(v22)--(v12)
(v22)--(v13)
(v22)--(v15)
(v15)--(v5)
(v15)--(v12)
(v15)--(v14)
(v14)--(v5)
(v14)--(v13)
(v13)--(v5)
(v13)--(v12);

\draw[rightedge]
(v41)--(v36)
(v41)--(v37)
(v41)--(v38)
(v41)--(v39)
(v41)--(v40)
(v40)--(v31)
(v40)--(v32)
(v40)--(v36)
(v40)--(v39)
(v39)--(v30)
(v39)--(v31)
(v39)--(v38)
(v38)--(v29)
(v38)--(v30)
(v38)--(v37)
(v37)--(v28)
(v37)--(v29)
(v37)--(v36)
(v36)--(v28)
(v36)--(v32)
(v32)--(v18)
(v32)--(v21)
(v32)--(v28)
(v32)--(v31)
(v31)--(v17)
(v31)--(v18)
(v31)--(v30)
(v30)--(v17)
(v30)--(v29)
(v29)--(v16)
(v29)--(v17)
(v29)--(v28)
(v28)--(v16)
(v28)--(v21)
(v21)--(v12)
(v21)--(v16)
(v21)--(v18)
(v18)--(v5)
(v18)--(v12)
(v18)--(v17)
(v17)--(v5)
(v17)--(v16)
(v16)--(v5)
(v16)--(v12);

\draw[crossedge] (v4)--(v13);
\draw[crossedge] (v6)--(v18);
\draw[crossedge] (v15)--(v16);

\foreach \i in {1,2,3,4,6,7,8,9,10,11,19,20}{
    \node[leftvertex] at (v\i) {};
}

\foreach \i in {13,14,15,22,23,24,25,26,27,33,34,35}{
    \node[middlevertex] at (v\i) {};
}

\foreach \i in {16,17,18,21,28,29,30,31,32,36,37,38,39,40,41}{
    \node[rightvertex] at (v\i) {};
}

\foreach \i in {5,12}{
    \node[sharedvertex] at (v\i) {};
}

\node at (50,4) {$e$};
\node at (45,52) {$f$};

\end{tikzpicture}
\caption{The unique smallest $4$-connected $5$-edge-connected planar graph $G_{41}$ that is a counterexample to Conjecture~\ref{conj:AB}. It has $41$ vertices. The blue vertices together with the two yellow vertices induce a copy of $Q-xy$, the green vertices together with the two yellow vertices induce another copy of $Q-xy$, and the red vertices together with the two yellow vertices induce a copy of $R-bc$. The edges $e$ and $f$ are shown in bold; together with the third thick black edge, they are the only edges not belonging to any of these three induced subgraphs.}
\label{fig:graph41}
\end{figure}

Moreover, it is not too difficult to describe an infinite family of planar $4$-connected $5$-edge-connected counterexamples to Conjecture~\ref{conj:AB}. We computationally verified that there are edges $e,f \in E(G_{41})$ (shown in bold in Fig.~\ref{fig:graph41}) such that $a(G_{41}-e)=a(G_{41}-f)=a(G_{41}-\{e,f\})=20$. If one deletes $e$ or $f$, a face of size $4$ is obtained. One can start with $G_{41}$, then repeatedly delete the edge corresponding to $f$ in the last added copy, place a new copy of $G_{41}-e$ inside the resulting face of size $4$ and add eight edges between the two facial cycles of length $4$ needed to triangulate the graph without creating a separating triangle (more precisely, label the facial cycles $x_1x_2x_3x_4x_1$ and $y_1y_2y_3y_4y_1$ in compatible cyclic order and add the edges $x_iy_i$ and $x_iy_{i+1}$ for $i\in\{1,2,3,4\}$, where indices are taken modulo $4$). If one performs this operation $k-1$ times, one obtains a planar triangulation $G_{41,k}$ on $41k$ vertices without a separating triangle and with minimum degree at least $5$ (therefore $G_{41,k}$ is $4$-connected and $5$-edge-connected, since the edge-connectivity of a planar triangulation equals its minimum degree~\cite{PPU23}). Moreover, every induced forest of $G_{41,k}$ contains at most $20$ vertices from each of the $k$ corresponding $41$-vertex subgraphs, so $a(G_{41,k})\leq 20k<\frac{41k}{2}$.

\section{An infinite family of graphs showing that $c_{\mathcal P} \leq 25/52$}
\label{sec:ratio2552}
We will now construct a planar graph $G_{52}$ on $52$ vertices such that every maximum induced forest in this graph has order $25$. The graph $G_{52}$ is obtained by taking the disjoint union of three copies $R_1, R_2, R_3$ of $R$ (having distinguished edges $b_ic_i$, $i \in \{1,2,3\}$), adding a new vertex $z$ and adding the edges $zb_i, zc_i, b_ic_{i+1},c_ic_{i+1}$ for every $i \in \{1,2,3\}$ (indices taken modulo $3$). Fig.~\ref{fig:52vertex} shows a planar embedding of $G_{52}$.

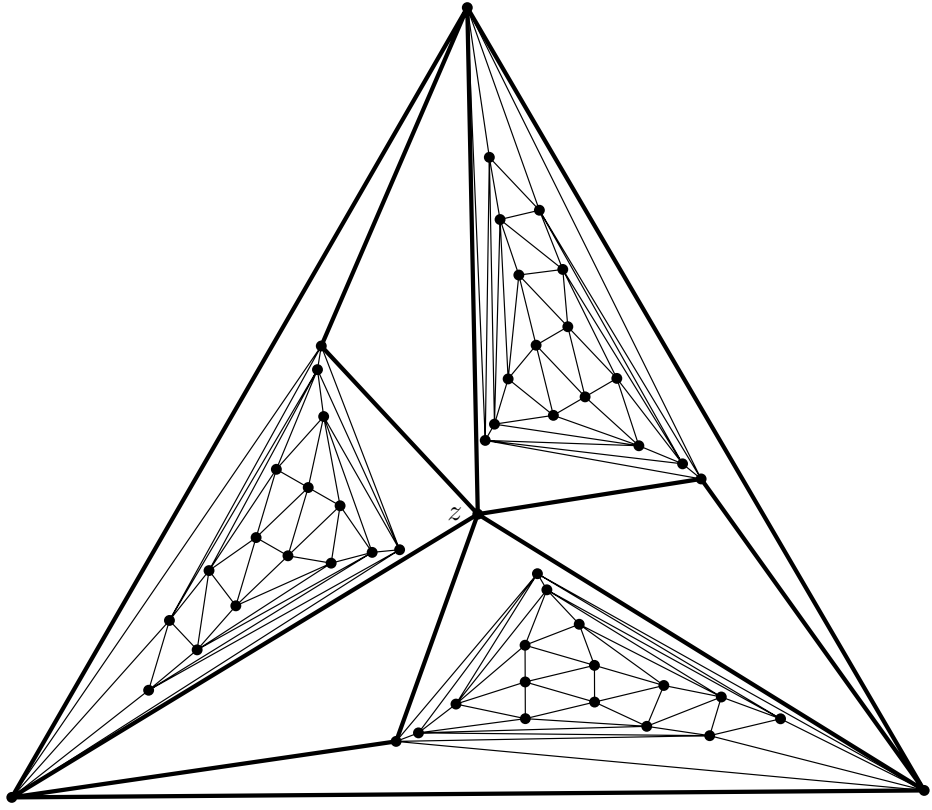
\begin{figure}
    \centering

\begin{tikzpicture}[xscale=4.3,yscale=4.3]

\tikzset{
    graphvertex/.style={
        circle,
        fill=black,
        inner sep=1.45pt
    },
    graphedge/.style={
        line width=0.45pt
    },
    connectingedge/.style={
        line width=1.6pt
    }
}

\coordinate (v0) at (0.015142857143,-0.322124447228);

\coordinate (v1)  at ( 0.702264673949,-0.213716840611);
\coordinate (v2)  at (-0.017142857143, 1.236373341684);
\coordinate (v3)  at ( 0.442886118845, 0.095644075126);
\coordinate (v4)  at ( 0.345281615251, 0.038780927541);
\coordinate (v5)  at ( 0.292159844261, 0.254627730633);
\coordinate (v6)  at ( 0.108495873640, 0.093991887598);
\coordinate (v7)  at ( 0.247677409432,-0.018083070979);
\coordinate (v8)  at ( 0.141433351497, 0.413610526817);
\coordinate (v9)  at ( 0.194555583897, 0.197763785904);
\coordinate (v10) at ( 0.050659585530, 0.775832490610);
\coordinate (v11) at ( 0.510733364870,-0.111622974856);
\coordinate (v12) at ( 0.276595944498, 0.430858804762);
\coordinate (v13) at ( 0.066427263703,-0.045203418659);
\coordinate (v14) at ( 0.083673802880, 0.584456711189);
\coordinate (v15) at ( 0.645164656595,-0.167260743948);
\coordinate (v16) at ( 0.038048116094,-0.095176539704);
\coordinate (v17) at ( 0.204613894431, 0.613028728585);

\coordinate (v18) at (-0.236162024998,-1.021224880751);
\coordinate (v19) at ( 1.388571428571,-1.172087627399);
\coordinate (v20) at ( 0.161441664520,-0.951276920702);
\coordinate (v21) at ( 0.160998985969,-0.838317367273);
\coordinate (v22) at ( 0.374488686267,-0.900235965648);
\coordinate (v23) at ( 0.327205950751,-0.660860379813);
\coordinate (v24) at ( 0.160555421599,-0.725357646258);
\coordinate (v25) at ( 0.587535072910,-0.849194391983);
\coordinate (v26) at ( 0.374045195755,-0.787276224291);
\coordinate (v27) at ( 0.946615378347,-0.951692986555);
\coordinate (v28) at (-0.051980489145,-0.906400834345);
\coordinate (v29) at ( 0.534891223281,-0.974872770135);
\coordinate (v30) at ( 0.227693584414,-0.554830241777);
\coordinate (v31) at ( 0.764371983024,-0.884596247756);
\coordinate (v32) at (-0.167379856451,-0.995002863497);
\coordinate (v33) at ( 0.198605165887,-0.505266618487);
\coordinate (v34) at ( 0.728646030151,-1.003619448073);

\coordinate (v35) at (-0.466405169959, 0.195232245224);
\coordinate (v36) at (-1.418571428571,-1.193516198827);
\coordinate (v37) at (-0.604630304374,-0.184076630562);
\coordinate (v38) at (-0.506583122228,-0.240173036406);
\coordinate (v39) at (-0.666951051537,-0.394101241123);
\coordinate (v40) at (-0.436004345400,-0.472840983923);
\coordinate (v41) at (-0.408535352040,-0.296268758901);
\coordinate (v42) at (-0.729270945415,-0.604125610973);
\coordinate (v43) at (-0.568903300661,-0.450197037752);
\coordinate (v44) at (-0.997577484885,-0.863848980193);
\coordinate (v45) at (-0.459055396734,-0.021685666936);
\coordinate (v46) at (-0.811789688787,-0.495695510764);
\coordinate (v47) at (-0.309851940554,-0.439675815702);
\coordinate (v48) at (-0.848348306913,-0.739569939571);
\coordinate (v49) at (-0.478087321153, 0.122554131307);
\coordinate (v50) at (-0.225384374418,-0.431552032232);
\coordinate (v51) at (-0.933562445590,-0.649118756651);

\draw[graphedge]
(v1)--(v2)
(v1)--(v15)
(v1)--(v16)
(v1)--(v17)
(v2)--(v10)
(v2)--(v16)
(v2)--(v17)
(v3)--(v4)
(v3)--(v5)
(v3)--(v11)
(v3)--(v12)
(v3)--(v15)
(v4)--(v5)
(v4)--(v7)
(v4)--(v9)
(v4)--(v11)
(v5)--(v8)
(v5)--(v9)
(v5)--(v12)
(v6)--(v7)
(v6)--(v8)
(v6)--(v9)
(v6)--(v13)
(v6)--(v14)
(v7)--(v9)
(v7)--(v11)
(v7)--(v13)
(v8)--(v9)
(v8)--(v12)
(v8)--(v14)
(v10)--(v13)
(v10)--(v14)
(v10)--(v16)
(v10)--(v17)
(v11)--(v13)
(v11)--(v15)
(v11)--(v16)
(v12)--(v14)
(v12)--(v15)
(v12)--(v17)
(v13)--(v14)
(v13)--(v16)
(v14)--(v17)
(v15)--(v16)
(v15)--(v17);

\draw[graphedge]
(v18)--(v19)
(v18)--(v32)
(v18)--(v33)
(v18)--(v34)
(v19)--(v27)
(v19)--(v33)
(v19)--(v34)
(v20)--(v21)
(v20)--(v22)
(v20)--(v28)
(v20)--(v29)
(v20)--(v32)
(v21)--(v22)
(v21)--(v24)
(v21)--(v26)
(v21)--(v28)
(v22)--(v25)
(v22)--(v26)
(v22)--(v29)
(v23)--(v24)
(v23)--(v25)
(v23)--(v26)
(v23)--(v30)
(v23)--(v31)
(v24)--(v26)
(v24)--(v28)
(v24)--(v30)
(v25)--(v26)
(v25)--(v29)
(v25)--(v31)
(v27)--(v30)
(v27)--(v31)
(v27)--(v33)
(v27)--(v34)
(v28)--(v30)
(v28)--(v32)
(v28)--(v33)
(v29)--(v31)
(v29)--(v32)
(v29)--(v34)
(v30)--(v31)
(v30)--(v33)
(v31)--(v34)
(v32)--(v33)
(v32)--(v34);

\draw[graphedge]
(v35)--(v36)
(v35)--(v49)
(v35)--(v50)
(v35)--(v51)
(v36)--(v44)
(v36)--(v50)
(v36)--(v51)
(v37)--(v38)
(v37)--(v39)
(v37)--(v45)
(v37)--(v46)
(v37)--(v49)
(v38)--(v39)
(v38)--(v41)
(v38)--(v43)
(v38)--(v45)
(v39)--(v42)
(v39)--(v43)
(v39)--(v46)
(v40)--(v41)
(v40)--(v42)
(v40)--(v43)
(v40)--(v47)
(v40)--(v48)
(v41)--(v43)
(v41)--(v45)
(v41)--(v47)
(v42)--(v43)
(v42)--(v46)
(v42)--(v48)
(v44)--(v47)
(v44)--(v48)
(v44)--(v50)
(v44)--(v51)
(v45)--(v47)
(v45)--(v49)
(v45)--(v50)
(v46)--(v48)
(v46)--(v49)
(v46)--(v51)
(v47)--(v48)
(v47)--(v50)
(v48)--(v51)
(v49)--(v50)
(v49)--(v51);

\draw[connectingedge]
(v0)--(v1)
(v0)--(v2)
(v0)--(v18)
(v0)--(v19)
(v0)--(v35)
(v0)--(v36)
(v1)--(v19)
(v2)--(v19)
(v2)--(v35)
(v2)--(v36)
(v18)--(v36)
(v19)--(v36);

\foreach \i in {0,...,51}{
    \node[graphvertex] at (v\i) {};
}

\node[left=2pt] at (v0) {$z$};

\end{tikzpicture}
\caption{The $52$-vertex graph $G_{52}$. The bold edges indicate edges that are not present in one of the three copies of $R$.}
    \label{fig:52vertex}
\end{figure}

\begin{thr}
\label{th:25over52}
There exists a planar graph $G_{52}$ on $52$ vertices with $a(G_{52})=25$. Hence, we have $c_{\mathcal P}\leq \frac{25}{52}$.
\end{thr}

\begin{proof}
Consider a vertex set $F \subseteq V(G_{52})$ such that $G_{52}[F]$ is a forest. Recall from Section~\ref{sec:minOrder} that $\bigl(p_R(\emptyset),p_R(\{b\}),p_R(\{c\}),p_R(\{b,c\})\bigr)=(8,8,8,9)$. Therefore, for each $i \in \{1,2,3\}$, we have $|V(R_i) \cap F| \leq 8$, unless $\{b_i,c_i\} \subseteq F$, in which case $|V(R_i) \cap F| \leq 9$.

Suppose $t$ is the number of copies $R_i$ for which $\{b_i,c_i\}\subseteq F$, and let $\varepsilon=1$ if $z\in F$ and $\varepsilon=0$ otherwise. Hence, we obtain $|F|\leq 24+t+\varepsilon$. We now show that $t+\varepsilon \leq 1$. If $z \in F$ and there exists an $i \in \{1,2,3\}$ such that $\{b_i,c_i\}\subseteq F$, then $\{z,b_i,c_i\}$ induces a triangle. If there exists an $i \in \{1,2,3\}$ such that $\{b_i,c_i\}\subseteq F$ and $\{b_{i+1},c_{i+1}\}\subseteq F$ (indices taken modulo $3$), then $\{b_i,c_i,c_{i+1}\}$ induces a triangle. Both cases contradict the fact that $G_{52}[F]$ is a forest and therefore $a(G_{52}) \leq 25$. By combining $z$ with a forest in each copy of $R_i$ that contains neither $b_i$ nor $c_i$, we conclude that $a(G_{52})=25$.
\end{proof}

By taking the disjoint union of $k \geq 1$ copies of $G_{52}$ and applying Observation~\ref{obs:ToPlanarTriangulation}, we obtain the following corollary.

\begin{cor}
\label{cor:infFam}
    For each integer $k \geq 1$, there exists a planar triangulation $G$ on $52k$ vertices for which $a(G) \leq 25k$.
\end{cor}

Based on the computations performed for Theorem~\ref{th:4Connected}, we also draw the following conclusion.
\begin{obs}
    Let $G$ be a planar $4$-connected $5$-edge-connected graph on at most $41$ vertices. Then $\frac{a(G)}{|V(G)|} \geq \frac{25}{52}$.
\end{obs}

Finally, we draw the reader's attention to a conjecture by Chappell and Pelsmajer~\cite{CP13}. They considered a stronger variant of the Albertson-Berman conjecture for induced forests of bounded degree. More precisely, for a graph $G$ and an integer $d \geq 2$, let $f_d(G)$ be the maximum order of an induced forest of $G$ with maximum degree at most $d$. They formulated the following conjecture.

\begin{conj}[Chappell and Pelsmajer~\cite{CP13}]
\label{conj:CP}
Let $d\geq 2$ be an integer and let $G$ be a planar graph on $n$ vertices. Then we have $f_d(G)>\frac{2dn}{4d+1}$.
\end{conj}

The infinite family of counterexamples to Conjecture~\ref{conj:AB} from Corollary~\ref{cor:infFam} also forms an infinite family of counterexamples to Conjecture~\ref{conj:CP} for each integer $d \geq 7$, since $25k<\frac{2 \cdot d \cdot52k}{4 \cdot d+1}$.

\section{Open problems}
\label{sec:openProblems}
We conclude this paper by listing several natural open problems that remain.

\begin{prob}
    The results discussed in this paper imply that $\frac{2}{5} \leq c_{\mathcal{P}} \leq \frac{25}{52}$. Improve these bounds.
\end{prob}

\begin{prob}
    Does there exist a $5$-connected counterexample to the Albertson-Berman conjecture (i.e., Conjecture~\ref{conj:AB})?
\end{prob}

\begin{prob}
    Conjecture~\ref{conj:CP} remains open for $2 \leq d \leq 6$: Is it true in these cases that $f_d(G)>\frac{2dn}{4d+1}$?
\end{prob}

\section*{Acknowledgements}
Wouter Cames van Batenburg is supported by the Belgian National Fund for Scientific Research (FNRS). Jan Goedgebeur is supported by Internal Funds of KU Leuven and a grant of the Research Foundation Flanders (FWO) with grant number G0AGX24N. Several of the computations for this work were carried out using the supercomputer infrastructure provided by the VSC (Flemish Supercomputer Center), funded by the Research Foundation Flanders (FWO) and the Flemish Government.

\section*{Declaration of generative AI and AI-assisted technologies}

During the preparation of this work, the authors used ChatGPT (OpenAI, GPT-5.6 Sol) for assistance with proofreading, improving the presentation and rewriting parts of the proofs, as well as with the preparation of figures and tables. The authors critically reviewed all AI-assisted output and edited it where necessary. All mathematical statements, arguments, and proofs in this paper were independently verified by the authors, who take full responsibility for the correctness and content of the manuscript.


\end{document}